\documentclass[11pt]{amsart}
\usepackage{amsmath,amssymb,amsthm}
\usepackage{graphicx}
\usepackage{xcolor}
\usepackage{hyperref}
\usepackage{cleveref}
\usepackage{tikz-cd}

\DeclareMathOperator{\tr}{tr}
\DeclareMathOperator{\supp}{supp}
\newcommand\sym{\mathrm{sym}}
\newcommand\abs[1]{\left\lvert{#1}\right\rvert}

\newtheorem{theorem}{Theorem}[section]
\newtheorem{lemma}[theorem]{Lemma}
\newtheorem{corollary}[theorem]{Corollary}
\newtheorem{proposition}[theorem]{Proposition}

\theoremstyle{definition}
\newtheorem{definition}[theorem]{Definition}

\newtheorem{notation}[theorem]{Notation}
\newtheorem{remark}[theorem]{Remark}

\title[Finite Element Spaces of Polynomial Double 2-Forms]{Finite Element Spaces of Double Two-Forms With Polynomial Coefficients}
\author{Yakov Berchenko-Kogan and Lily DiPaulo}

\begin{document}

\begin{abstract}
We develop finite element spaces of symmetric tensor products of two-forms with polynomial coefficients. In three dimensions, these give higher order finite element spaces of matrix fields with normal--normal continuity, which have applications to the TDNNS method for elasticity, for example. In general dimension, these spaces can be used to represent the Riemann curvature tensor in numerical relativity. In many ways, our methods parallel Li's work generalizing Regge calculus to higher order, as Regge elements can be thought of as symmetric tensor products of one-forms. However, whereas the constant coefficient Regge space has one shape function per edge, the constant coefficient space of double-forms in our paper has one shape function per triangle and two shape functions per tetrahedron, so we must address the fact that there are shape functions of two different types. Like Li, we obtain an explicit geometrically decomposed basis of shape functions.
\end{abstract}

\maketitle

\section{Introduction}

Motivated by problems in elasticity and in numerical geometry and relativity, there has been recent interest \cite{bega25double, hu2025finite} in developing finite element spaces for tensor products of differential forms, which are known as double forms or form-valued forms. We denote these spaces $\Lambda^{p,q}:=\Lambda^p\otimes\Lambda^q$ and refer to them as $(p,q)$-forms. One can think of these double forms as $(p+q)$-tensors that are antisymmetric in the first $p$ indices and antisymmetric in the last $q$ indices.

A well-studied special case is $p=q=1$, in which case $(1,1)$-forms are simply bilinear forms, which can also be thought of as matrix fields with tangential--tangential interelement continuity. Finite element spaces for symmetric bilinear forms with constant coefficients on simplicial triangulations were developed by Christiansen \cite{christiansen2004characterization, christiansen2011linearization} based on the discrete formulation of Regge \cite{re61}. These spaces were generalized to higher order polynomial coefficients by Li \cite{li2018regge} and are now known as \emph{Regge elements}. (Note that \emph{antisymmetric} bilinear forms are the same as $2$-forms, so their finite element spaces fall under the purview of finite element exterior calculus \cite{Arnold_Falk_Winther_2006}.)

In this paper, we study finite element spaces of symmetric $(2,2)$-forms. In three dimensions, $(2,2)$-forms are equivalent to matrix fields with normal--normal interelement continuity, which are used in Pechstein's and Sch{\"o}berl's Tangential Displacement, Normal--Normal Stress (TDNNS) method for elasticity \cite{pechstein2011tangential}. Meanwhile, in geometry, the Riemann curvature tensor is a symmetric $(2,2)$-form, so there is a need for these spaces in numerical geometry \cite{gopalakrishnan2023analysis}. With constant coefficients, finite element spaces for symmetric $(2,2)$-forms were developed in three dimensions in \cite{pechstein2011tangential} (as symmetric matrix fields), and in general dimension in \cite{bega25double, hu2025finite}.

In this paper, we develop finite element spaces for $(2,2)$-forms with higher order polynomial coefficients. Such higher order spaces were already considered in both \cite{bega25double} and \cite{hu2025finite}, so our main contribution is to give explicit geometrically decomposed computational bases for these spaces, analogous to Li's natural bases for higher order symmetric $(1,1)$-forms (bilinear forms) \cite{li2018regge}. We note, however, that our construction is self-contained and does not rely on the machinery in either \cite{bega25double} or \cite{hu2025finite}. Additionally, we expect these methods to generalize to higher order spaces of $(p,q)$-forms more generally, using the constant-coefficient spaces of \cite{bega25double} as a starting point.

In Section~\ref{sec:l22}, we set notation and discuss one key difference between $(2,2)$-forms and $(1,1)$-forms. In general, $(p+q)$-forms $\Lambda^{p+q}$ embed into the space of $(p,q)$-forms $\Lambda^{p,q}$, since $(p+q)$-tensors that are antisymmetric in all indices are, in particular, antisymmetric in the first $p$ indices and in the last $q$ indices. As discussed in \cite{bega25double}, the $\Lambda^{p+q}$ subspace of $\Lambda^{p,q}$ is somewhat exceptional and is best treated separately using finite element exterior calculus \cite{Arnold_Falk_Winther_2006}. In our case of interest, we see that $4$-forms are \emph{symmetric} $(2,2)$-forms, whereas $2$-forms are \emph{antisymmetric} $(1,1)$-forms. As a result, we do not need to worry about this exceptional case when studying symmetric $(1,1)$-forms, but we do need to consider it in the study of symmetric $(2,2)$-forms. We find that the space of symmetric $(2,2)$-forms naturally decomposes as $\Lambda^{2,2}_\sym=\Lambda^{2,2}_0\oplus\Lambda^4$, where $\Lambda^{2,2}_0$ consists of those $(2,2)$-forms that satisfy the algebraic Bianchi identity. Because finite element exterior calculus gives finite element spaces for $\Lambda^4$, the task of this paper is to develop finite element spaces for $\Lambda^{2,2}_0$. Additionally, note that, with regards to the aforementioned applications, the Bianchi identity always holds in dimension three, and the Bianchi identity holds for the Riemann curvature tensor.

In Li's work on symmetric bilinear forms \cite{li2018regge}, Li leverages the constant coefficient spaces in order to construct spaces with higher order polynomial coefficients. We adopt the same approach in this paper, and so we begin by discussing the constant coefficient space on a simplex in Section~\ref{sec:simplex}. There is, however, a key difference: For symmetric bilinear forms with constant coefficients, there is just one type of shape function, associated to edges of the triangulation. In contrast, for the space $\Lambda^{2,2}_0$, there are two types of shape function, with one type associated to triangles and the other type associated to tetrahedra. 

Finally, in Section~\ref{sec:poly}, we develop a finite element space for $\mathcal P_r\Lambda^{2,2}_0$ of forms in $\Lambda^{2,2}_0$ with polynomial coefficients. As we will see, the fact that the constant coefficient space has shape functions associated to faces of different dimensions complicates matters. Nonetheless, analogous to Li's result for bilinear forms \cite{li2018regge}, we obtain a geometrically decomposed basis of shape functions for $\mathcal P_r\Lambda^{2,2}_0$.

\section{Decomposition of $\Lambda^{2,2}_{\sym}$}\label{sec:l22}

In this section, we establish definitions and notation, and we show that the space $\Lambda^{2,2}_{\sym}$ of symmetric $(2,2)$-forms decomposes as a sum of the space $\Lambda^{2,2}_0$ of $(2,2)$-forms that satisfy the Bianchi identity and the space $\Lambda^4$ of fully antisymmetric $4$-tensors. For decompositions of double form spaces in greater generality, see \cite[Section~2]{bega25double}, especially Section~2.4.4 for its application to $(2,2)$-forms.

\begin{definition}
For $k\ge 0$, let $\Lambda^k(\mathbb R^n)$ denote the space of constant coefficient differential $k$-forms on $\mathbb R^n$, that is, alternating $k$-linear forms on $\mathbb R^n$ with constant coefficients. This space is spanned by wedge products of the form $\alpha_{1} \wedge \alpha_{2} \wedge \cdots\wedge\alpha_k$, where each $\alpha_i\in\Lambda^1(\mathbb R^n)$.

Let $\mathcal P_r\Lambda^k(\mathbb R^n)$ denote the space of differential $k$-forms with polynomial coefficients of degree at most $r$. Likewise, let $C^\infty\Lambda^k(\mathbb R^n)$ denote those with smooth coefficients. When there is no danger of ambiguity, we will omit the $\mathbb R^n$ for each of these spaces.
\end{definition}

We now shift our focus onto the space of \emph{double forms}. 

\begin{definition}
    Let \[\Lambda^{2,2}_{\sym}=\Lambda^2\odot\Lambda^2\] denote the space of \emph{double $2$-forms} on $\mathbb R^n$, where $\odot$ denotes the symmetric tensor product, given by \[\alpha\odot\beta=\alpha\otimes\beta+\beta\otimes\alpha\] for any $\alpha,\beta\in\Lambda^2$. Equivalently, $\Lambda^{2,2}_{\sym}$ can be viewed as the space of $4$-tensors that are antisymmetric in the first two indices, antisymmetric in the last two indices, and symmetric with respect to swapping the first and last pairs of indices. That is, for any $\omega\in\Lambda^{2,2}_{\sym}$ we have \[\omega(x,y;z,w)=-\omega(y,x;z,w)=-\omega(x,y;w,z)=\omega(z,w;x,y)\]
    for all $x,y,z,w\in\mathbb R^n$.
\end{definition}

\begin{definition}
    We define the \emph{wedge map} \[\wedge:\Lambda^{2,2}_{\sym}\to\Lambda^4\]
    by \[\wedge(\alpha\odot\beta):=\alpha\wedge\beta\] for any $\alpha,\beta\in\Lambda^2$. 
    We also define \[\Lambda_0^{2,2}:=\ker(\wedge)\] as the subspace of symmetric double $2$-forms that map to zero under the wedge map.
\end{definition}

We note that $\Lambda^{2,2}_0$ is the space of symmetric $(2,2)$-forms that satisfy the Bianchi identity; see \cite[Sections~2.4.3 and 2.4.1]{bega25double}.

\begin{lemma}\label{lemma:inclusion_map}
    Every element of $\Lambda^4$ satisfies the symmetries of $\Lambda^{2,2}_{\sym}$. 
\end{lemma}
\begin{proof}
    If $\omega\in \Lambda^4$, then $\omega$ is fully antisymmetric in all four of its indices. In particular, we have that $\omega(x,y,z,w)=-\omega(y,x,z,w)$ and $\omega(x,y,z,w)=-\omega(x,y,w,z)$. Swapping the first and last pairs of indices requires two swaps, which is a positive permutation, so $\omega(z,w,x,y)=\omega(x,y,z,w)$.
\end{proof}

Consequently, $\Lambda^4$ can be identified with a subspace of $\Lambda^{2,2}_{\sym}$. We make this identification explicit with an inclusion map.

\begin{definition}
    Let
    \[i:\Lambda^4\hookrightarrow\Lambda^{2,2}_{\sym}\] be the natural inclusion map defined by \[(i\omega)(x,y;z,w):=\omega(x,y,z,w),\] for all $x,y,z,w\in\mathbb R^n$. 
\end{definition}

\begin{proposition}\label{prop:wedge-inclusion}
For any $\omega \in \Lambda^4$, the wedge map and inclusion map satisfy 
\[
\wedge\bigl(i(\omega)\bigr) = 3\omega.
\]
\end{proposition}
\begin{proof}
It suffices to verify on a spanning set. Let $\alpha_1,\alpha_2,\alpha_3,\alpha_4\in\Lambda^1$ and set
\[
\omega := \alpha_1 \wedge \alpha_2 \wedge \alpha_3 \wedge \alpha_4.
\]
As a tensor, $\omega$ corresponds to the fully antisymmetric combination of all $24$ index permutations of the $\alpha_i$, namely
\[
\omega = \sum_{\sigma \in S_4} \mathrm{sign}(\sigma) \cdot 
\alpha_{\sigma(1)} \otimes \alpha_{\sigma(2)} \otimes \alpha_{\sigma(3)} \otimes  \alpha_{\sigma(4)},
\]
where $S_4$ is the permutation group on $\{1,2,3,4\}$. We can now sort these as follows into three groups of $8$ terms, each of which corresponds to a symmetric double-form:
\begin{align*}
(\alpha_1\wedge\alpha_2) \odot (\alpha_3\wedge\alpha_4)
&=
\phantom{+} \alpha_1 \otimes \alpha_2 \otimes \alpha_3 \otimes \alpha_4
+ \alpha_3 \otimes \alpha_4 \otimes \alpha_1 \otimes \alpha_2 \\
&\quad
- \alpha_1 \otimes \alpha_2 \otimes \alpha_4 \otimes \alpha_3
- \alpha_4 \otimes \alpha_3 \otimes \alpha_1 \otimes \alpha_2 \\
&\quad
- \alpha_2 \otimes \alpha_1 \otimes \alpha_3 \otimes \alpha_4
- \alpha_3 \otimes \alpha_4 \otimes \alpha_2 \otimes \alpha_1 \\
&\quad
+ \alpha_2 \otimes \alpha_1 \otimes \alpha_4 \otimes \alpha_3
+ \alpha_4 \otimes \alpha_3 \otimes \alpha_2 \otimes \alpha_1 \\[1.5ex]
(\alpha_1\wedge\alpha_3) \odot (\alpha_4\wedge\alpha_2)
&=
\phantom{+} \alpha_1 \otimes \alpha_3 \otimes \alpha_4 \otimes \alpha_2
+ \alpha_4 \otimes \alpha_2 \otimes \alpha_1 \otimes \alpha_3 \\
&\quad
- \alpha_1 \otimes \alpha_3 \otimes \alpha_2 \otimes \alpha_4
- \alpha_4 \otimes \alpha_2 \otimes \alpha_3 \otimes \alpha_1 \\
&\quad
- \alpha_3 \otimes \alpha_1 \otimes \alpha_4 \otimes \alpha_2
- \alpha_2 \otimes \alpha_4 \otimes \alpha_1 \otimes \alpha_3 \\
&\quad
+ \alpha_3 \otimes \alpha_1 \otimes \alpha_2 \otimes \alpha_4
+ \alpha_2 \otimes \alpha_4 \otimes \alpha_3 \otimes \alpha_1 \\[1.5ex]
(\alpha_1\wedge\alpha_4) \odot (\alpha_2\wedge\alpha_3)
&=
\phantom{+} \alpha_1 \otimes \alpha_4 \otimes \alpha_2 \otimes \alpha_3
+ \alpha_2 \otimes \alpha_3 \otimes \alpha_1 \otimes \alpha_4 \\
&\quad
- \alpha_1 \otimes \alpha_4 \otimes \alpha_3 \otimes \alpha_2
- \alpha_3 \otimes \alpha_2 \otimes \alpha_1 \otimes \alpha_4 \\
&\quad
- \alpha_4 \otimes \alpha_1 \otimes \alpha_2 \otimes \alpha_3
- \alpha_2 \otimes \alpha_3 \otimes \alpha_4 \otimes \alpha_1 \\
&\quad
+ \alpha_4 \otimes \alpha_1 \otimes \alpha_3 \otimes \alpha_2
+ \alpha_3 \otimes \alpha_2 \otimes \alpha_4 \otimes \alpha_1
\end{align*}
So
\[
i(\omega) = (\alpha_1\wedge\alpha_2) \odot (\alpha_3\wedge\alpha_4)
+ (\alpha_1\wedge\alpha_3) \odot (\alpha_4\wedge\alpha_2)
+ (\alpha_1\wedge\alpha_4)\odot (\alpha_2\wedge\alpha_3).
\]
It follows that
\begin{align*}
\wedge\bigl(i(\omega)\bigr)
&=(\alpha_1\wedge\alpha_2\wedge\alpha_3\wedge\alpha_4) 
 + (\alpha_1\wedge\alpha_3\wedge\alpha_4\wedge\alpha_2) 
 + (\alpha_1\wedge\alpha_4\wedge\alpha_2\wedge\alpha_3) \\
&= \omega + \omega + \omega \\
&= 3\omega. \qedhere
\end{align*}
\end{proof}

\begin{corollary}\label{cor:wedge_surjective}
The wedge map 
\[
\wedge : \Lambda^{2,2}_{\sym} \longrightarrow \Lambda^4
\]
is surjective.
\end{corollary}

\begin{proposition}
\label{prop:decomposition_sym22}
The space of symmetric double $2$-forms admits a natural decomposition as 
\[
\Lambda^{2,2}_{\sym}
=
\Lambda_0^{2,2}
\oplus
\Lambda^4.
\]
Here, with a slight abuse of notation, we identify $\Lambda^4$ with its image under the inclusion map $i:\Lambda^4\hookrightarrow \Lambda^{2,2}_{\sym}$ and treat it as a subspace of $\Lambda^{2,2}_{\sym}$.
\end{proposition}
\begin{proof}
By definition,  $\Lambda_0^{2,2} $ is the kernel of the wedge map
\[
\wedge : \Lambda^{2,2}_{\sym} \to \Lambda^4.
\]
By Corollary~\ref{cor:wedge_surjective}, the wedge map is surjective, and we therefore have the split exact sequence
\[
\begin{tikzcd}[column sep=2.5em]
0 \arrow[r] 
& \Lambda_0^{2,2} \arrow[r, hook] 
& \Lambda^{2,2}_{\sym} 
  \arrow[r, shift left=0.7ex, "\wedge"] 
  \arrow[r, shift right=0.7ex, "\frac{1}{3}i"', leftarrow]
& \Lambda^4 \arrow[r] 
& 0
\end{tikzcd}
\]
The claim then follows by the splitting lemma.
\end{proof}

\begin{corollary}
\label{corollary:dimension_counting}
The space $\Lambda^{2,2}_{\sym}$ has dimension
\[
\dim\bigl(\Lambda^{2,2}_{\sym}\bigr) = \tfrac{1}{8}n(n-1)(n^2-n+2),
\]
and the subspace $\Lambda_0^{2,2}$ has dimension
\[
\dim\Lambda_0^{2,2} = \tfrac{1}{12}n^2(n+1)(n-1).
\]
\end{corollary}

\begin{proof}
The space $\Lambda^2$ of $2$-forms on an $n$-dimensional vector space has dimension $\binom{n}{2}$, so the space $\Lambda^{2,2}_{\sym}$ of symmetric products of $2$-forms has dimension
\[
\binom{\binom n2+1}2=\tfrac18 n(n-1)(n^2-n+2).
\]
The space $\Lambda^4$ of fully antisymmetric $4$-forms has dimension $\binom{n}{4}$. 
Since the wedge map is surjective by Corollary~\ref{cor:wedge_surjective}, the kernel $\Lambda_0^{2,2}$ has dimension
\[
\dim\Lambda_0^{2,2} = \binom{\binom{n}{2} + 1}{2} - \binom{n}{4}.
\]
So,
\begin{align*}
\dim\Lambda_0^{2,2}
&= \tfrac1{24}n(n-1)\bigl(3(n^2-n+2) - (n-2)(n-3)\bigr)\\
&= \tfrac{1}{12}n^2(n+1)(n-1).\qedhere
\end{align*}
\end{proof}

See also \cite[Lemma~5.16]{bega25double}.

\section{Constant coefficient spaces}\label{sec:simplex}
We now study the constant coefficient spaces on the standard simplex in $\mathbb R^{n+1}$. 
\begin{definition}\label{def:simplex}
Let
\[
T^n = \Bigl\{\lambda = (\lambda_0,\lambda_1,\dots,\lambda_n)\in\mathbb R^{n+1} :
\lambda_i \ge 0,\ \sum_{i=0}^n \lambda_i = 1\Bigr\}
\]
be the standard $n$-dimensional simplex, with barycentric coordinates $\lambda_i$.

The numbers $\{0,\dots,n\}$ index the vertices of $T^n$: vertex $i$ is the vertex satisfying $\lambda_i=1$. For a face $F$ of $T^n$, let $I(F)$ denote the subset of $\{0,\dotsc,n\}$ corresponding to the vertices of $F$.

Given a form or double form $\varphi$ on $T^n$ and face $F$ of $T^n$, we can pull back $\varphi$ to $F$ via the inclusion map $F\hookrightarrow T$; we call the resulting form on $F$ the \emph{trace} of $\varphi$ and denote it $\tr_F\varphi$.
\end{definition}

We will construct a geometrically decomposed basis of $\Lambda^{2,2}_0(T^n)$, with basis elements associated to faces of $T^n$, analogous to the geometric decompositions of differential forms \cite{Arnold_2009} and bilinear forms \cite{li2018regge}. See also \cite[Section~5.3.3]{bega25double}.

The differentials $d\lambda_i$ of the barycentric coordinates are one-forms on $T^n\subset\mathbb R^{n+1}$; they span $\Lambda^1(T^n)$ but are not linearly independent because $\sum_{i=0}^nd\lambda_i=0$. 
Likewise, the wedge products $d\lambda_i\wedge d\lambda_j$ span the space $\Lambda^2(T^n)$ but are not linearly independent.

\begin{notation}
    For convenience, we use the shorthand $d\lambda_{ij}:=d\lambda_i\wedge d\lambda_j.$ Similarly, $d\lambda_{ijkl}=d\lambda_i\wedge d\lambda_j\wedge d\lambda_k\wedge d\lambda_l$.
\end{notation}

\begin{definition}\label{def:betagamma}
For an $n$-dimensional simplex with barycentric coordinate functions $\lambda_i$, define for distinct indices $i,j,k$ the symmetric double form
\[
\beta_{ijk} := d\lambda_{ij}\odot d\lambda_{jk} + d\lambda_{jk}\odot d\lambda_{ki} + d\lambda_{ki}\odot d\lambda_{ij},
\]
which is associated with the triangular face determined by vertices $i,j,k$, denoted $T_{ijk}$.

Likewise, for distinct indices $i,j,k,l$, define
\[
\gamma_{ijkl} := d\lambda_{ik}\odot d\lambda_{lj} - d\lambda_{il}\odot d\lambda_{jk},
\]
which is naturally associated with the tetrahedral face determined by vertices $i,j,k,l$, denoted $T_{ijkl}$.
\end{definition}

\begin{proposition}\label{prop:beta_gamma_kernel}
The beta forms $\beta_{ijk}$ and the gamma forms $\gamma_{ijkl}$ belong to $\Lambda^{2,2}_0(T^n)$. 
In other words, both $\beta_{ijk}$ and $\gamma_{ijkl}$ wedge to zero in $\Lambda^4(T^n)$.
\end{proposition}

\begin{proof}
A form is in $\Lambda^{2,2}_0(T^n)$ if and only if its wedge vanishes.

Consider
\[
\beta_{ijk} = d\lambda_{ij}\odot d\lambda_{jk} + d\lambda_{jk}\odot d\lambda_{ki} + d\lambda_{ki}\odot d\lambda_{ij}.
\]
Then
\[
\wedge\bigl(\beta_{ijk}\bigr)
:=
\bigl(d\lambda_{ij}\wedge d\lambda_{jk}\bigr)
+
\bigl(d\lambda_{jk}\wedge d\lambda_{ki}\bigr)
+
\bigl(d\lambda_{ki}\wedge d\lambda_{ij}\bigr),
\]
which is trivially zero, as $d\lambda_l\wedge d\lambda_l=0$ for all $l\in\{i,j,k\}$. 

Now consider
\[
\gamma_{ijkl}=d\lambda_{ik}\odot d\lambda_{lj}-d\lambda_{il}\odot d\lambda_{jk}.
\]
Then, we have
\[
\begin{split}
\wedge\bigl(\gamma_{ijkl}\bigr)
&:=
\bigl(d\lambda_{ik} \wedge d\lambda_{lj}\bigr)
-\bigl(d\lambda_{il} \wedge d\lambda_{jk}\bigr)\\
 &\phantom{:}= -d\lambda_{ilkj} + d\lambda_{ilkj} = 0.
 \end{split}
\]
Therefore both the beta forms and the gamma forms all lie in $\Lambda^{2,2}_0(T^n)$.
\end{proof}

\begin{lemma}\label{lemma:vanishing_faces}
The beta forms $ \beta_{ijk} $ vanish on any face not containing $ T_{ijk} $, and likewise the $\gamma_{ijkl} $ vanish on any face not containing $ T_{ijkl} $.
\end{lemma}
\begin{proof}
Without loss of generality, suppose the face $F$ does not contain the vertex $i$. Then $\lambda_i = 0$  on $F$, so $d\lambda_i = 0$, and subsequently any expression involving $d\lambda_i$ vanishes on $F$. It follows that $\beta_{ijk} = 0$ on $F$. Likewise, $\gamma_{ijkl}=0$ on $F$.
\end{proof}

\begin{proposition}
Each form $\beta_{ijk}$ is fully symmetric in its indices:
\[
\beta_{ijk} = \beta_{jki} = \beta_{kij} = \beta_{ikj} = \beta_{jik} = \beta_{kji}.
\]
Meanwhile, each form $ \gamma_{ijkl} $ is symmetric under swapping the first two indices, symmetric under swapping the last two indices, and symmetric under swapping the first and last index pairs. Explicitly,
\begin{align*}
&\gamma_{ijkl} = \gamma_{jikl} = \gamma_{ijlk} = \gamma_{jilk} \\
={} &\gamma_{klij} = \gamma_{klji} = \gamma_{lkij} = \gamma_{lkji}.
\end{align*}
Furthermore, the $\gamma$-forms satisfy the cyclic identity
\[
\gamma_{ijkl} + \gamma_{iklj} + \gamma_{iljk} = 0.
\]
\end{proposition}
\begin{proof}
One can verify these identities directly from the definitions of $\beta_{ijk}$ and $\gamma_{ijkl}$, using the antisymmetry of the wedge product, and the symmetry of the symmetric tensor product.
\end{proof}

Thus, for any set of three vertices $\{i,j,k\}$, the corresponding $\beta$-forms span a space of dimension at most one, and for any set of four vertices $\{i,j,k,l\}$, the corresponding $\gamma$-forms span a space of dimension at most two. In the remainder of this section, we will show that the dimensions of these spaces are, in fact, one and two, respectively. 

\begin{lemma}
\label{lemma:2d_mathring}
Let $ T_{ijk} $ be a triangular face spanned by vertices $ i,j,k $, and define $ \mathring{\Lambda}_0^{2,2}(T_{ijk})$ to be the subspace of $\Lambda_0^{2,2}(T_{ijk})$ of symmetric double $2$-forms vanishing on the boundary of the triangle. The space $\mathring \Lambda_0^{2,2}(T_{ijk})$ is one-dimensional and is spanned by $ \beta_{ijk} $.
\end{lemma}

\begin{proof}
On a triangle, the space $ \Lambda^2(T_{ijk}) $ of constant $2$-forms is one-dimensional and is equal to its trace-free subspace:
\[
\Lambda^2(T_{ijk}) = \mathring{\Lambda}^2(T_{ijk}),
\]
since every $2$-form vanishes on the edges. It follows that
\[
\Lambda^{2,2}_{\sym}(T_{ijk}) = \mathring{\Lambda}_{\mathrm{sym}}^{2,2}(T_{ijk}),
\]
is also one-dimensional. Furthermore, since any four-form is zero in two dimensions, it follows that
\[\mathring\Lambda^{2,2}_0(T_{ijk})= \mathring{\Lambda}_{\mathrm{sym}}^{2,2}(T_{ijk}).\]

So, it remains to show that $\beta_{ijk}$ is nonzero on the triangle $T_{ijk}$. On $T_{ijk}$, we have the identity,
\[
d\lambda_i + d\lambda_j + d\lambda_k = 0,
\]
from which it follows that
\[
d\lambda_{ij} = d\lambda_{jk} = d\lambda_{ki}.
\]
Let us denote this nonzero 2-form  $\omega$. It follows that, on $T_{ijk}$
\[
\beta_{ijk} = d\lambda_{ij} \odot d\lambda_{jk}
+ d\lambda_{jk} \odot d\lambda_{ki}
+ d\lambda_{ki} \odot d\lambda_{ij} = 3\omega \odot \omega \ne 0,
\]
and therefore spans $ \mathring{\Lambda}_0^{2,2}(T_{ijk}) $.
\end{proof}

\begin{proposition}
\label{prop:trace_surjective}
Let $T_{ijkl}$ be the tetrahedral face spanned by vertices $i,j,k,l$.  
We define $\mathring{\Lambda}^{2,2}_0(T_{ijkl})$ to be the subspace of $\Lambda^{2,2}_0(T_{ijkl})$ consisting of symmetric double $2$-forms vanishing on the boundary of the tetrahedron.

The trace map
\[
\tr : \Lambda_0^{2,2}(T_{ijkl}) \longrightarrow 
\Lambda_0^{2,2}(T_{ijk}) \oplus 
\Lambda_0^{2,2}(T_{ijl}) \oplus 
\Lambda_0^{2,2}(T_{ikl}) \oplus 
\Lambda_0^{2,2}(T_{jkl})
\]
is surjective. Consequently, the space $\mathring{\Lambda}_{0}^{2,2}(T_{ijkl})$ is two-dimensional.
\end{proposition}

\begin{proof}
Consider the four forms
\[
\beta_{ijk}, \quad \beta_{ijl}, \quad \beta_{ikl}, \quad \beta_{jkl} \in \Lambda_{0}^{2,2}(T_{ijkl}).
\]
By Lemma~\ref{lemma:2d_mathring}, when restricted to $T_{ijk}$, the form $\beta_{ijk}$ spans the one-dimensional space $ \Lambda_0^{2,2}(T_{ijk}) $.  By Lemma ~\ref{lemma:vanishing_faces}, $\beta_{ijk}$ vanishes on any face that does not contain the triangle $T_{ijk}$. Similarly, $\beta_{ijl},\beta_{ikl},$ and $\beta_{jkl}$ vanish on all triangular faces other than $T_{ijl},T_{ikl},$ and $T_{jkl}$, respectively.

Therefore, applying the trace map, we have
\[
\tr(\beta_{ijk}) = (\beta_{ijk}, 0, 0, 0), \quad
\tr(\beta_{ijl}) = (0, \beta_{ijl}, 0, 0), \quad \text{etc}.
\]
Since the image of each $\beta_{ijk}$ spans the double forms on its associated face, it follows that the trace map is surjective.

Since four-forms are zero in three dimensions, $\Lambda^{2,2}_0(T_{ijkl})=\Lambda^{2,2}_{\sym}(T_{ijkl})$, which has dimension $6$. Since the domain of the trace map has dimension $6$ and the codomain has dimension $ 4 $, the kernel is $2$-dimensional.
\end{proof}

\begin{lemma}
\label{lemma:gamma_independence}
On a tetrahedron with vertices $i,j,k,l$, the forms $\gamma_{iklj}$ and $\gamma_{iljk}$ are linearly independent.
Consequently, they are a basis for $\mathring{\Lambda}_{0}^{2,2}(T_{ijkl}).$ 

\end{lemma}
\begin{proof}

From Definition~\ref{def:betagamma}, we have that
\[
\gamma_{iklj} := d\lambda_{il}\odot d\lambda_{jk} - d\lambda_{ij}\odot d\lambda_{kl}, \qquad
\gamma_{iljk} := d\lambda_{ij}\odot d\lambda_{kl}-d\lambda_{ik}\odot d\lambda_{lj}.
\]
We choose vectors tangent to the tetrahedron
\[
V_1 = e_i - e_l,\quad W_1 = e_j - e_k,\quad V_2 = e_i - e_k,\quad W_2 = e_l - e_j.
\]
We plug these vectors into each term of the $\gamma$-forms, obtaining
\begin{align*}
(d\lambda_{ij} \odot d\lambda_{kl})(V_1, W_1; V_1, W_1) &= -2, &\quad
(d\lambda_{ij} \odot d\lambda_{kl})(V_2, W_2; V_2, W_2) &= 2, \\
(d\lambda_{il} \odot d\lambda_{jk})(V_1, W_1; V_1, W_1) &= 0,  &\quad
(d\lambda_{il} \odot d\lambda_{jk})(V_2, W_2; V_2, W_2) &= -2, \\
(d\lambda_{ik} \odot d\lambda_{lj})(V_1, W_1; V_1, W_1) &= 2,  &\quad
(d\lambda_{ik} \odot d\lambda_{lj})(V_2, W_2; V_2, W_2) &= 0.
\end{align*}
It follows that
\begin{align*}
\gamma_{iklj}(V_1, W_1; V_1, W_1) &= 2,  &\quad
\gamma_{iklj}(V_2, W_2; V_2, W_2) &= -4, \\
\gamma_{iljk}(V_1, W_1; V_1, W_1) &= -4, &\quad
\gamma_{iljk}(V_2, W_2; V_2, W_2) &= 2.
\end{align*}

Since the vectors $\begin{pmatrix}2&-4 
\end{pmatrix}$ and $\begin{pmatrix}
    -4&2
\end{pmatrix}$ are linearly independent, it follows that the $\gamma$-forms are linearly independent on $T_{ijkl}$. By Proposition~\ref{prop:beta_gamma_kernel}, both forms lie in $\Lambda_0^{2,2}$, and by Lemma~\ref{lemma:vanishing_faces}, they vanish on all two-dimensional faces of the tetrahedron. Therefore, they are elements of the subspace $\mathring{\Lambda}_{0}^{2,2}(T_{ijkl})$. Finally, by Proposition~\ref{prop:trace_surjective}, this space has dimension two, so $\gamma_{iklj}$ and $\gamma_{iljk}$ form a basis for this space.
\end{proof}

We are now ready to write down a geometrically decomposed basis for $\Lambda^{2,2}_0(T^n)$.

\begin{remark}
We will use the prefix $\mathcal B$ to denote our basis of the corresponding space. Note, however, that Hu and Lin \cite{hu2025finite} use this symbol to denote bubble spaces, that is, spaces of forms with vanishing trace. In our paper, the bubble spaces are denoted with a ring.
\end{remark}

\begin{definition}
For $T^n$ an $n$-dimensional simplex, let
\begin{multline*}
\mathcal{B}\Lambda_0^{2,2}(T^n) \\:= 
\bigl\{ \beta_{ijk} : 0\le i<j<k\le n \bigr\} \cup
\bigl\{ \gamma_{iklj}, \gamma_{iljk} : 0\le i<j<k<l\le n \bigr\}.
\end{multline*}
\end{definition}

We will show that $\mathcal B\Lambda_0^{2,2}(T^n)$ is a basis for $\Lambda_0^{2,2}(T^n)$.

\begin{proposition}
\label{prop:basis_independence}
The set
$
\mathcal{B}\Lambda_0^{2,2}(T^n)
$
is linearly independent.
\end{proposition}
\begin{proof}
Suppose
\[
\sum_{i<j<k} c_{ijk}  \beta_{ijk}
+
\sum_{i<j<k<l} \left( a_{ijkl}  \gamma_{iklj} + b_{ijkl}  \gamma_{iljk} \right)
= 0.
\]

Now, we can take the trace of this identity to a 2-dimensional face $T_{ijk}$. By Lemma~\ref{lemma:vanishing_faces}, each $\gamma$-form vanishes on every 2-dimensional face, and each $\beta_{i'j'k'}$ vanishes on $T_{ijk}$ unless $\{i',j',k'\}=\{i,j,k\}$, in which case its trace is nonzero by Lemma~\ref{lemma:2d_mathring}. It follows that $c_{ijk} = 0$, and this holds for every $i<j<k$.

Therefore, we have
\[
\sum_{i<j<k<l} \left( a_{ijkl}  \gamma_{iklj} + b_{ijkl}  \gamma_{iljk} \right) = 0.
\]
Similarly, we can take the trace to a 3-dimensional face $T_{ijkl}$. The forms $\gamma_{iklj}$ and $\gamma_{iljk}$ are linearly independent on $T_{ijkl}$ by Lemma~\ref{lemma:gamma_independence}, and every other $\gamma$-form in the sum vanishes on $T_{ijkl}$ by Lemma~\ref{lemma:vanishing_faces}, so $a_{ijkl} = b_{ijkl} = 0$ for all $i<j<k<l$.

All of the coefficients vanish, so the set $\mathcal{B}\Lambda_0^{2,2}(T^n)$ is linearly independent.
\end{proof}

\begin{proposition}
\label{prop:basis_count}
The set $\mathcal{B}\Lambda_{0}^{2,2}(T^n)$ has $\frac1{12}n^2(n+1)(n-1)$ elements.
\end{proposition}

\begin{proof}
The set $\mathcal{B}\Lambda_0^{2,2}(T^n)$ has one beta-form associated with each 2-dimensional face and two gamma-forms associated with each 3-dimensional face, so
\begin{align*}
\left\lvert\mathcal{B}\Lambda_{0}^{2,2}(T^n)\right\rvert&=\binom{n+1}{3} + 2\binom{n+1}{4}\\
&= \tfrac1{12}(n+1)n(n-1)\bigl(2 + (n-2)\bigr) 
= \tfrac{1}{12}n^2(n+1)(n-1).\qedhere
\end{align*}
\end{proof}

\begin{proposition}
\label{prop:basis-constant}
The set
\[
\mathcal{B}\Lambda_0^{2,2}(T^n) = 
\bigl\{ \beta_{ijk} : i<j<k \bigr\} \cup
\bigl\{ \gamma_{iklj}, \gamma_{iljk} : i<j<k<l \bigr\},
\]
where $\{i,j,k,l\}\subseteq\{0,\dots,n\}$,
is a basis for $\Lambda_0^{2,2}(T^n)$.
\end{proposition}

\begin{proof}
By Proposition~\ref{prop:basis_independence}, the set $\mathcal{B}\Lambda_0^{2,2}(T^n)$ is linearly independent.  
By Corollary~\ref{corollary:dimension_counting} and Proposition~\ref{prop:basis_count}, its cardinality equals $\dim\Lambda_0^{2,2}(T^n)$.  
Hence it is a basis.
\end{proof}

Applying Proposition~\ref{prop:basis-constant} to a face $F$ of $T^n$, we obtain the following corollary. Recall that $I(F)$ denotes the set of vertices of $F$.

\begin{corollary}\label{cor:basis_face_constant}
Let $F$ be a face of $T^n$. Then a basis of $\Lambda_{0}^{2,2}(F)$ is
\[
\mathcal B\Lambda_{0}^{2,2}(F)
:=
\bigl\{ \beta_{ijk} : i<j<k,\bigr\}
\ \cup\
\bigl\{ \gamma_{iklj},\ \gamma_{iljk} : i<j<k<l\bigr\},
\]
where $\{i,j,k,l\}\subseteq I(F)$. 
\end{corollary}

Note that, in the above corollary, we view the $\beta_{ijk}$, $\gamma_{iklj}$, and $\gamma_{iljk}$ as double forms on $F$, not $T^n$. We make the relationship more explicit in the next corollary. We first establish some notation.

\begin{notation}\label{not:Ibetagamma}
For a basis element $\psi\in\mathcal{B}\Lambda_0^{2,2}(T^n)$ we denote by $I(\psi)$ its set of associated indices:
\[
I(\beta_{ijk})=\{i,j,k\}, \qquad
I(\gamma_{iklj})=I(\gamma_{iljk})=\{i,j,k,l\}.
\]
\end{notation}

\begin{corollary}\label{cor:trace_basis_face}
Let $F$ be a face of $T^n$.  
Then
\[
\mathcal B\Lambda_{0}^{2,2}(F)
= \{\tr_F\psi : \psi \in \mathcal B\Lambda_{0}^{2,2}(T^n),\ I(\psi)\subseteq I(F)\}.
\]
\end{corollary}

\section{Polynomial coefficient spaces and geometric decomposition}\label{sec:poly}
We now move on to double forms with polynomial coefficients. We begin with notation.

\begin{definition}
    Let $\mathcal P_r\Lambda^{2,2}_0(T^n)$ denote the space of double forms on $T^n$ that are pointwise in $\Lambda^{2,2}_0(T^n)$ and have polynomial coefficients of degree at most $r$. Equivalently, $\mathcal P_r\Lambda^{2,2}_0$ is spanned by products of polynomials of degree at most $r$ and constant coefficient double forms in $\Lambda^{2,2}_0(T^n)$
\end{definition}

\begin{notation}
Let $\mathcal P_r(T^n)$ denote the space of polynomials on $T^n$ of degree at most $r$. Each such polynomial can be written as a \emph{homogeneous} polynomial of degree $r$ in the barycentric coordinates. A monomial of degree $r$ is indexed by a multi-index
\[
\alpha=(\alpha_0,\alpha_1,\dots,\alpha_n)\in\mathbb N_0^{n+1},
\qquad\abs\alpha=\alpha_0+\alpha_1+\dots+\alpha_n=r,
\]
and is written
\[
\lambda^\alpha:=\lambda_0^{\alpha_0}\lambda_1^{\alpha_1}\cdots\lambda_n^{\alpha_n}.
\]
The \emph{support} of $\alpha$ is $\supp(\alpha)=\{i:\alpha_i\neq0\}$.
\end{notation}

Observe that
    \[\mathcal P_r\Lambda_0^{2,2}(T^n)\cong\mathcal P_r(T^n)\otimes \Lambda_0^{2,2}(T^n).\]
The monomials are a basis for $\mathcal P_r(T^n)$, and Proposition~\ref{prop:basis-constant} gives a basis $\mathcal{B}\Lambda^{2,2}_0(T^n)$ for $\Lambda^{2,2}_0(T^n)$. Therefore, we immediately obtain a basis $\mathcal{BP}_r\Lambda^{2,2}_0(T^n)$ for $\mathcal P_r\Lambda^{2,2}_0(T^n)$.

\begin{proposition}
\label{prop:basis_polynomial}
For $r\ge 0$,
\[
\mathcal B \mathcal P_r\Lambda_{0}^{2,2}(T^n)
:=\{\lambda^\alpha\psi : \psi\in\mathcal B\Lambda_{0}^{2,2}(T^n),\ \alpha\in\mathbb N_0^{n+1},\ \abs\alpha=r\}
\]
is a basis of $\mathcal P_r\Lambda_{0}^{2,2}(T^n)$.
\end{proposition}

Each basis element in $\mathcal{BP}_r\Lambda^{2,2}_0(T^n)$ is naturally associated to a subset of $\{0,\dots,n\}$.

\begin{notation}\label{not:Ipoly}
Recall in Notation~\ref{not:Ibetagamma}, we established that $I(\beta_{ijk})=\{i,j,k\}$ and $I(\gamma_{iklj})=I(\gamma_{iljk})=\{i,j,k,l\}$. Similarly, for a monomial $\lambda^\alpha$, we set
\[
I(\lambda^\alpha) := \supp(\alpha),
\]
and for $\varphi=\lambda^\alpha\psi$ with $\psi\in\mathcal B\Lambda_{0}^{2,2}(T^n)$, we set
\[
I(\varphi) := I(\lambda^\alpha)\cup I(\psi).
\]
Recall also that $I(F)$ denotes the set of vertices of a face $F$ of $T^n$.
\end{notation}

We thus have a natural correspondence between basis elements in $\mathcal{BP}_r\Lambda^{2,2}_0(T^n)$ and faces of $T^n$. As we will see later, this correspondence yields the geometric decomposition. 

For now, applying Proposition~\ref{prop:basis_polynomial} to a face $F$ of $T^n$, we obtain the following corollary.

\begin{corollary}\label{cor:basis_face_polynomial}
Let $F$ be a face of $T^n$.  
Then a basis of $\mathcal P_r\Lambda_{0}^{2,2}(F)$ is
\[
\mathcal B \mathcal P_r\Lambda_{0}^{2,2}(F)
:=\{\lambda^\alpha\psi : \psi\in\mathcal B\Lambda_{0}^{2,2}(F),\ 
\alpha\in\mathbb N_0^{n+1},\ \abs\alpha=r,\ \supp(\alpha)\subseteq I(F)\}.
\]
Equivalently,
\[
\mathcal B\mathcal P_r\Lambda_{0}^{2,2}(F)
=\bigl\{\tr_F\varphi : \varphi\in\mathcal B\mathcal P_r\Lambda_{0}^{2,2}(T^n),\ I(\varphi)\subseteq I(F)\bigr\}.
\]
\end{corollary}

We now prove a relationship between the vertices associated to a basis element $\varphi$ and its traces to the faces of $T^n$.

\begin{proposition}\label{prop:4_trace-vanish} 
Let $\varphi\in\mathcal B\mathcal P_r\Lambda_{0}^{2,2}(T^n)$, and let $F$ be a face of $T^n$. Then the trace of $\varphi$ to $F$ is nonzero if and only if
\[
I(\varphi)\subseteq I(F).
\]
\end{proposition}

\begin{proof}
Let $\varphi=\lambda^\alpha\psi$, with $\alpha\in\mathbb N_0^{n+1}$, $\abs\alpha=r$, and $\psi\in\mathcal B\Lambda_{0}^{2,2}(T^n)$.

Assume $I(\varphi)\not\subseteq I(F)$.  
Then there is an index $i$ in $I(\varphi)=I(\lambda^\alpha)\cup I(\psi)$ but not in $I(F)$.  
If $i\in I(\lambda^\alpha)$, then $\lambda_i=0$ on $F$, so $\lambda^\alpha=0$ on $F$ and hence $\varphi=0$.  
If $i\in I(\psi)$, then $\psi$ vanishes on $F$ by Lemma ~\ref{lemma:vanishing_faces}, so again $\varphi=0$.  

Conversely, if $I(\varphi)=I(\lambda^\alpha)\cup I(\psi)\subseteq I(F)$, then both $\lambda^\alpha$ and $\psi$ have nonzero trace to $F$, and therefore $\varphi$ has nonzero trace to $F$ as well.
\end{proof}

\begin{corollary}\label{cor:trace-vanish-boundary}
Let $\varphi\in \mathcal B{\mathcal P}_r \Lambda_{0}^{2,2}(T^{n})$.  
Then $\varphi$ vanishes on the boundary of $T^n$ if and only if
\[
I(\varphi)=\{0,1,\dots,n\}.
\]
\end{corollary}

\begin{proof}
If $I(\varphi)\neq\{0,\dots,n\}$, then let $i$ be an index not in $I(\varphi)$.  
Let $F$ be the face opposite vertex $i$, so $I(F)=\{0,\dots,n\}\setminus\{i\}$. Then $I(\varphi)\subseteq I(F)$, so Proposition~\ref{prop:4_trace-vanish} implies that $\varphi$ has nonzero trace to $F$. Hence $\varphi$ does not vanish on the boundary of $T^n$.

Conversely, if $I(\varphi)=\{0,\dots,n\}$, then for any proper face $F$ of $T^n$ we have
$I(\varphi)\nsubseteq I(F)$.
By Proposition~\ref{prop:4_trace-vanish}, it follows that $\tr_F\varphi=0$.
We conclude that $\varphi$ vanishes on the boundary of $T^n$.
\end{proof}

\begin{definition}\label{def:mathring}
Let $\mathring{\mathcal P}_r\Lambda^{2,2}_0(T^n)$ be the subset of $\mathcal P_r\Lambda^{2,2}_0(T^n)$ consisting of double forms with vanishing trace on the boundary of $T^n$.
\end{definition}

\begin{proposition}\label{prop:basis_mathring}
Let $r\ge 0$. Then
\[
\mathcal B \mathring{\mathcal P}_r \Lambda_{0}^{2,2}(T^{n})
:=\bigl\{
\varphi\in\mathcal B\mathcal P_r\Lambda_{0}^{2,2}(T^{n})
:I(\varphi)=\{0,\dots,n\}\bigr\}
\]
is a basis of $\mathring{\mathcal P}_r \Lambda_{0}^{2,2}(T^{n})$.
\end{proposition}

\begin{proof}
Elements of $\mathcal B\mathring{\mathcal P}_r\Lambda^{2,2}_0(T^n)$ are in $\mathring{\mathcal P}_r\Lambda^{2,2}_0(T^n)$ by Corollary~\ref{cor:trace-vanish-boundary}, and linear independence follows from the fact that $\mathcal B\mathring{\mathcal P}_r\Lambda^{2,2}_0(T^n)$ is contained in the basis $\mathcal B\mathcal P_r\Lambda_0^{2,2}(T^n)$. It remains to show that $\mathcal B\mathring{\mathcal P}_r\Lambda^{2,2}_0(T^n)$ spans $\mathring{\mathcal P}_r\Lambda^{2,2}_0(T^n)$.

 Let $\sigma\in\mathring{\mathcal P}_r\Lambda_{0}^{2,2}(T^n)$. We can expand $\sigma$ in the basis of $\mathcal P_r\Lambda_{0}^{2,2}(T^n)$ as
\[
\sigma=\sum_{\varphi\in\mathcal B\mathcal P_r\Lambda_{0}^{2,2}(T^n)} c_{\varphi}\,\varphi,
\qquad \text{for } c_\varphi\in\mathbb R.
\]
Let $i\in\{0,\dots,n\}$ and let $F$ be the face opposite vertex $i$, so $I(F)=\{0,\dots,n\}\setminus\{i\}$. We will now take the trace of above equation to $F$.
Since $\sigma$ has vanishing trace, we have $\tr_F\sigma=0$. Using Proposition~\ref{prop:4_trace-vanish},
only those basis terms with $I(\varphi)\subseteq I(F)$ have nonzero trace to $F$, hence
\[
0=\tr_F\sigma
=\sum_{\substack{\varphi\in\mathcal B\mathcal P_r\Lambda_{0}^{2,2}(T^n)\\ I(\varphi)\subseteq I(F)}}
c_\varphi\,\tr_F\varphi,
\]
By Corollary~\ref{cor:basis_face_polynomial},
the set $\{\tr_F\varphi:\varphi\in\mathcal B\mathcal P_r\Lambda_{0}^{2,2}(T^n), I(\varphi)\subseteq I(F)\}$ is a basis of
$\mathcal P_r\Lambda_{0}^{2,2}(F)$, so all of the corresponding coefficients 
$c_\varphi$ are zero. In other words, we showed that $c_\varphi$ is zero whenever $I(\varphi)$ does not contain $i$.

Repeating this for each $i\in\{0,\dots,n\}$ eliminates every term except those with
$I(\varphi)=\{0,\dots,n\}$. Therefore $\sigma$ is a linear combination of the proposed basis elements in $\mathcal B \mathring{\mathcal P}_r \Lambda_{0}^{2,2}(T^{n})$, and so we conclude that $\mathcal B \mathring{\mathcal P}_r \Lambda_{0}^{2,2}(T^{n})$ spans $\mathring{\mathcal P}_r\Lambda_{0}^{2,2}(T^n)$, as desired.
\end{proof}

We now determine the dimension of $\mathring{\mathcal P}_r\Lambda_{0}^{2,2}(T^n)$.

\begin{proposition}\label{prop:counting_mathring}
Let $r \ge 0$. The set $\mathcal B \mathring{\mathcal P}_r \Lambda_{0}^{2,2}(T^{n})$ has
\[
\binom{n+1}{3}\binom{r+2}{n}
+
2\binom{n+1}{4}\binom{r+3}{n}
\]
elements. In particular, for $r < n - 3$, this set has zero elements.
\end{proposition}

\begin{proof}
Every basis element in $\mathcal B \mathring{\mathcal P}_r \Lambda_{0}^{2,2}(T^{n})$ can be written as $
\varphi = \lambda^\alpha \psi$, where $
\abs\alpha = r$, $\psi \in \mathcal B\Lambda^{2,2}_0(T^n)$,
and every index in $\{0,\dots,n\}$ appears either in $\psi$ or in $\lambda^\alpha$. 

Suppose first that $\psi = \beta_{ijk}$ with $i<j<k$. Then for any $l\notin\{i,j,k\}$, we have that $\lambda_l$ appears with exponent at least one in $\lambda^\alpha$. Therefore we can write $\lambda^\alpha=\mu\lambda^{\alpha'}$ where $\mu$ is the product of the $\lambda_l$ for $\lambda\in\{0,\dotsc,n\}\setminus\{i,j,k\}$. In particular, $\deg \mu = n - 2$. The number of possible $\lambda^{\alpha'}$ is given by the number of monomials of degree $r - (n - 2)$ in $n+1$ variables, which is $\binom{r + 2}{n}$.
Meanwhile, there are $\binom{n+1}{3}$ choices of $\psi$ of the form $\beta_{ijk}$ corresponding to the choice of the three indices $\{i,j,k\}$. So, in total, there are $\binom{n+1}{3}\binom{r + 2}{n}$ choices for $\psi$ of this form.

When $\psi=\gamma_{ikjl}$ or $\psi=\gamma_{iljk}$, the argument is similar. In this case $\mu$ has degree $n - 3$, so we obtain $\binom{r + 3}{n}$ possible $\lambda^{\alpha'}$ for each choice of $\psi$. There are  $2\binom{n+1}{4}$ choices of $\psi$ because there are two choices of $\psi$ for every $\binom{n+1}{4}$ choices of the indices $\{i,j,k,l\}$. So, in total, there are $2\binom{n+1}{4}\binom{r + 3}{n}$ choices for $\psi$ of this form.

Putting these together, the statement follows.
\end{proof}

We now state the claims analogous to Propositions~\ref{prop:basis_mathring} and \ref{prop:counting_mathring} for faces $F$ of $T^n$.

\begin{corollary}
\label{cor:basis_mathring_face}
Let $F$ be a face of $T^n$ of dimension $m$. 
Then the set
\[
\mathcal B\mathring{\mathcal P}_r\Lambda_{0}^{2,2}(F):=\{\varphi : \varphi\in\mathcal B\mathcal P_r\Lambda_0^{2,2}(F), I(\varphi)=I(F) \}
\]
is a basis of $\mathring{\mathcal P}_r\Lambda_0^{2,2}(F)$. Moreover, \[
\abs{\mathcal B\mathring{\mathcal P}_r\Lambda_{0}^{2,2}(F)}
=
\binom{m+1}{3}\binom{r+2}{m}
+
2\binom{m+1}{4}\binom{r+3}{m}.
\]
\end{corollary}

We are now ready to discuss the geometric decomposition of $\mathcal P_r\Lambda^{2,2}_0(T^n)$. We begin by discussing the extension operators that naturally arise from the formulas for $\beta_{ijk}$ and $\gamma_{ijkl}$.

\begin{definition}\label{def:extension}
Let $F$ be a face of $ T^n$. We define the extension operator
\[
E_F:\mathring{\mathcal P}_r\Lambda_{0}^{2,2}(F)\to\mathcal P_r\Lambda_{0}^{2,2}(T^n)
\]
on basis elements by
\[
E_F(\lambda^\alpha\beta_{ijk})=\lambda^\alpha\beta_{ijk},\quad \text{and} \quad
E_F(\lambda^\alpha\gamma_{ijkl})=\lambda^\alpha\gamma_{ijkl}.
\]
where the forms on the left-hand side are viewed as forms on $F$, and those on the right are viewed as forms on $T^n$.

\end{definition}

\begin{theorem}\label{thm:extension_direct_sum}
For $r\ge0$, the basis $\mathcal B \mathcal P_r\Lambda_0^{2,2}(T^n)$ decomposes as the disjoint union
\[
\mathcal B\mathcal P_r\Lambda_{0}^{2,2}(T^n)
=\bigsqcup_{F\subseteq T^n}E_F\bigl(\mathcal B\mathring{\mathcal P}_r\Lambda_{0}^{2,2}(F)\bigr),
\]
 where $\mathcal B\mathcal P_r\Lambda_0^{2,2}(T^n)$ is the basis from Proposition~\ref{prop:basis_polynomial}, and $\mathcal B\mathring{\mathcal P}_r\Lambda_{0}^{2,2}(F)$ is the basis of the vanishing trace space from Corollary~\ref{cor:basis_mathring_face}. It follows that
\[
\mathcal P_r\Lambda_{0}^{2,2}(T^n)
=\bigoplus_{F\subseteq T^n}E_F\bigl(\mathring{\mathcal P}_r\Lambda_{0}^{2,2}(F)\bigr).
\]
\end{theorem}
\begin{proof}
By Corollary~\ref{cor:basis_mathring_face},  $\mathcal B\mathring{\mathcal P}_r\Lambda_{0}^{2,2}(F)$ consists of the set of $\varphi$ with $I(\varphi)=I(F)$. By Definition~\ref{def:extension}, we then have
\[
E_F\bigl(\mathcal B\mathring{\mathcal P}_r\Lambda_{0}^{2,2}(F)\bigr)
=\{\varphi\in\mathcal B\mathcal P_r\Lambda_{0}^{2,2}(T^n):I(\varphi)=I(F)\}.
\]
Since every $\varphi\in\mathcal{BP}_r\Lambda^{2,2}_0(T^n)$ satisfies $I(\varphi)=I(F)$ for exactly one $F$, we see that the subsets $E_F\bigl(\mathcal B\mathring{\mathcal P}_r\Lambda_{0}^{2,2}(F)\bigr)$ are disjoint and their union is all of $\mathcal B\mathcal P_r\Lambda_{0}^{2,2}(T^n)$. 
\end{proof}

So far, we have worked on a single simplex, but, following \cite{be2025extension}, we can define a global space $\mathcal P_r\Lambda^{2,2}_0(\mathcal T)$ on a triangulation $\mathcal T$, where the interelement continuity conditions are that the traces (in the sense of Definition~\ref{def:simplex}) to an interface from both sides match. Then, using \cite{be2025extension}, the analogous geometric decomposition result follows from our definition of extension operators and Proposition~\ref{prop:4_trace-vanish}.

\section{Conclusion and Future Work}

We have developed a geometrically decomposed basis for the finite element space $\mathcal P_r\Lambda^{2,2}_0(\mathcal T)$ of double two-forms with polynomial coefficients satisfying the Bianchi identity, analogous to Li's work \cite{li2018regge} on the space $\mathcal P_r\Lambda^{1,1}_0(\mathcal T)$ of symmetric bilinear forms. The space $\mathcal P_r\Lambda^{2,2}_0(\mathcal T)$ has $\binom{m+1}{3}\binom{r+2}{m}+2\binom{m+1}{4}\binom{r+3}{m}$ degrees of freedom corresponding to each face of dimension $m$, where $\binom{m+1}{3}\binom{r+2}{m}$ of them correspond to shape functions that are monomial multiples of the $\beta_{ijk}$, and $2\binom{m+1}{4}\binom{r+3}{m}$ of them correspond to shape functions that are monomial multiples of the $\gamma_{ijkl}$. The fact that there are shape functions of two different types is the key difference between our work on double two-forms and Li's work on double one-forms.

Going forward, one could generalize our work further to finite element spaces $\mathcal P_r\Lambda^{p,q}_0(\mathcal T)$ of arbitrary double forms with polynomial coefficients. With constant coefficients $r=0$, these spaces were already explored in detail in \cite{bega25double}, where one finds that, in general, there are $\min\{p,q\}$ types of double forms to consider. (As discussed in \cite{bega25double}, in general, there are many natural subspaces of double forms, which are denoted there by $\Lambda^{p,q}_m$, but it suffices to consider the $m=0$ case because $\Lambda^{p,q}_m\cong\Lambda^{p+m,q-m}_0$.) In our work, for each specific type ($\beta_{ijk}$ or $\gamma_{ijkl}$) of double form in the basis for $\Lambda^{2,2}_0(\mathcal T)$, the construction of the geometrically decomposed basis worked much the same way as in Li \cite{li2018regge}, and we expect the same techniques to work in general by dealing with each of the $\min\{p,q\}$ types of double forms in the basis for $\Lambda^{p,q}_0(\mathcal T)$ individually.

\section*{Acknowledgements}
This work was supported by NSF award DMS-2411209. We would also like to thank Evan Gawlik for helpful feedback.

\bibliographystyle{plain}
\bibliography{refs}

\end{document}